\documentclass[12pt,a4paper]{amsart}
\usepackage[all]{xy}
\usepackage{amscd}
\usepackage{amsfonts}
\usepackage{amsmath}
\usepackage{amssymb}
\usepackage{amsthm}
\usepackage{bbm}
\usepackage[colorlinks,linkcolor=blue]{hyperref}
\usepackage{enumerate}
\usepackage{extarrows}
\usepackage{graphicx}
\usepackage{graphics}
\usepackage{mathrsfs}
\usepackage[mathscr]{eucal}
\usepackage[margin=2.9cm]{geometry}

\theoremstyle{plain}
\newtheorem{thm}{Theorem}[subsection]
\newtheorem{lema}[thm]{Lemma}

\newtheorem{ppst}[thm]{Proposition}

\theoremstyle{definition}
\newtheorem{defn}[thm]{Definition}
\newtheorem{rem}[thm]{Remark}
\newtheorem{exam}[thm]{Example}
\newtheorem{conj}{Conjecture}

\title[]{Examples on a Conjecture about Makar-Limanov Invariants of Affine Unique Factorization Domains}
\author{Ziqi Liu}

\address{Jilin University, Changchun, Jilin, China}
\email{liuzq0616@mails.jlu.edu.cn}

\thanks{This paper is an undergraduate research work. I thank professor Xiaosong Sun in Jilin University for his help and suggestions when I worked on this paper.}

\begin{document}

\begin{abstract}
  The author introduces a conjecture about Makar-Limanov invariants of affine unique factorization domains over a field of characteristic zero. Then the author finds that the conjecture does not always hold when $\mathbbm{k}$ is not algebraically closed and gives some examples where the conjecture holds.
\end{abstract}

\maketitle
\tableofcontents

\newpage
\section{Introduction}
Throughout the work, $A^{[n]}$ denotes the polynomial ring over a given ring $A$ with $n$ variables; $\mathbbm{k}$ denotes a field of characteristic zero; $\mathbb{A}^n$ denotes the affine space over $\mathbbm{k}$ of dimension $n$. An \textbf{affine $\mathbbm{k}$-variety} is an irreducible algebraic set over $\mathbbm{k}$; given an affine variety $V\subseteq\mathbb{A}^n$, define $\mathcal{I}(V):=\{f\in\mathbbm{k}^{[n]}:f(a)=0,\forall a\in V\}$; given a subring $I$ of $\mathbb{A}^{[n]}$, define $\mathcal{Z}(I):=\{a\in\mathbb{A}^{n}:f(a)=0,\forall f\in I\}$.

\vspace{3mm}
In paper \cite{ML01-1}, Makar-Limanov provides several conjunctures about Makar-Limanov invariant of an affine unique factorization domain. One of his conjectures is that
\begin{conj}
If $A$ is an affine UFD over $\mathbb{C}$, then $\mathcal{ML}(A)=\mathcal{ML}(A^{[1]})$.
\end{conj}
Here we are going to discuss a generalized version of this conjecture which originally put forward on complex field $\mathbb{C}$.
\begin{conj}\label{conj}
If $A$ is an affine UFD over $\mathbbm{k}$, then $\mathcal{ML}(A)=\mathcal{ML}(A^{[1]})$.
\end{conj}
Based on following parts, one can see that $\mathcal{ML}(A^{[1]})\subseteq\mathcal{ML}(A)$ and $\mathbbm{k}\subseteq\mathcal{ML}(A)$. Therefore, we only need to focus on $\mathcal{ML}(A^{[1]})\supseteq\mathcal{ML}(A)$ when $\mathcal{ML}(A)\neq\mathbbm{k}$.

\vspace{3mm}
In Section 2, the author introduces the definition of an affine unique factorization domain and some methods to identify affine unique factorization domains. Also, the author reviews the definition of a locally nilpotent derivation over a $\mathbbm{k}$-algebra and then introduce the Makar-Limanov invariant of a $\mathbbm{k}$-algebra.

\vspace{3mm}
In Section 3, the author lists some important results about Conjecture \ref{conj}. Based on those results, we only need to consider affine UFD $A$ over $\mathbbm{k}$ with $\mathcal{ML}(A)\neq A$ and $\mathcal{ML}(A)\neq\mathbbm{k}$. Then, the author proves that Conjecture \ref{conj} holds for Danielewski domains when $\mathbbm{k}$ is algebraically closed and is false when $\mathbbm{k}=\mathbb{R}$. Also, the author proves that Conjecture \ref{conj} is always true for Koras-Russell threefolds. Moreover, the author checks this conjecture on affine UFDs constructed by Finston and Maubach. In this paper, one of such affine UFDs is called a Finston-Maubach domain.

\vspace{3mm}
In section 4, the author makes several comments on this paper as well as the research topic.

\newpage
\section{Preliminaries}
\subsection{Affine unique factorization domains}
\begin{defn}
If a nonzero ring (with multiplicative identity) $R$ has no nonzero zero divisors, then it is a \textbf{domain}. If the domain $R$ is also commutative, then it is an \textbf{integral domain}.
\end{defn}
\begin{exam}
The ring of integers $\mathbb{Z}$ is an integral domain while the Hurwitz integer $\mathbb{H}$ is a domain but not an integral domain.
\end{exam}
\begin{defn}
Given a $\mathbbm{k}$-vector space $A$ and a binary operation $\cdot:A\times A\rightarrow A$, denote the product of $x$ and $y$ in $\mathbbm{k}$ by $xy$ here and the addition in $A$ by $+$, if\\
(1) $(\textbf{a}+\textbf{b})\cdot\textbf{c} = \textbf{a}\cdot\textbf{c} + \textbf{b}\cdot\textbf{c}$;\\
(2) $\textbf{c}\cdot(\textbf{a}+\textbf{b}) = \textbf{c}\cdot\textbf{a} + \textbf{c}\cdot\textbf{b}$;\\
(3) $(x\textbf{a})(y\textbf{b})=(x y)(\textbf{a}\cdot\textbf{b})$;\\
holds for all $\textbf{a},\textbf{b},\textbf{c}\in A$ and $x,y\in\mathbbm{k}$, then the pair $(A,\cdot)$ is an \textbf{algebra over $\mathbbm{k}$}, or simply a \textbf{$\mathbbm{k}$-algebra}, also denoted by $A$.
\end{defn}
\begin{defn}
If an integral domain $A$ containing $\mathbbm{k}$ is finitely generated as a $\mathbbm{k}$-algebra, then $A$ is an \textbf{affine domain over $\mathbbm{k}$}.
\end{defn}

\begin{thm} \quad\\
(1) Given an affine $\mathbbm{k}$-variety $V$, its coordinate ring $\mathbbm{k}[V]:=\mathbb{A}^{[n]}/\mathcal{I}(V)$ is an affine domain over $\mathbbm{k}$.\\
(2) If $\mathbbm{k}$ is algebraically closed, and let $A$ be an affine domain over $\mathbbm{k}$. Then there exist an affine $\mathbbm{k}$-variety, such that $A\cong\mathbbm{k}[V]$.
\begin{proof} (1) We know that $I:=\mathcal{I}(V)$ is prime since $V$ is irreducible, therefore $\mathbbm{k}[V]:=\mathbb{A}^{[n]}/I$ is an integral domain. Note that $\mathbbm{k}[V]$ is clearly a finitely generated $\mathbbm{k}$-algebra, it is an affine domain over $\mathbbm{k}$.\\
(2) Choose generators $a_1,\dots,a_n$ of $A$, the surjective homomorphism $F:\mathbbm{k}^{[n]}\rightarrow A$ sending $f$ to $f(a_1,\dots,a_n)$ yields $A\cong\mathbbm{k}^{[n]}/I$ where $I=\textup{Ker}(F)$. Since $A$ is an integral domain, $I$ is a prime ideal. Let $V:=\mathcal{Z}(I)$, we know $I=\mathcal{I}(V)$ from the Hilbert's Nullstellensatz, so $A\cong\mathbbm{k}[V]$.
\end{proof}
\end{thm}
\begin{rem}
Some very details of this proof can be found in book \cite{Ke05}. One can see that if $\mathbbm{k}$ is algebraically closed, then there is a bijection between the set of all affine domains over $\mathbbm{k}$ and the set of all affine varieties over $\mathbbm{k}$.
\end{rem}

\begin{defn}
If each element $x\neq0$ of an integral domain $A$ is a product of a unit in $A$ and prime elements $p_i$ in $A$, that is
\begin{center}$x= u\cdot p_1\cdot \dots \cdot p_m$\end{center}
then $A$ is a \textbf{unique factorization domain}, abbreviated to a UFD. If a UFD is also an affine domain over $\mathbbm{k}$, then it is a \textbf{affine unique factorization domain over $\mathbbm{k}$}, abbreviated to an affine UFD over $\mathbbm{k}$. In addition, an affine $\mathbbm{k}$-variety $X$ is called \textbf{factorial} if its coordinate ring $\mathbbm{k}[X]$ is an affine UFD.
\end{defn}
\begin{rem}
Since proving an integral domain to be a UFD is hard in general, it is natural to find special methods to tell an affine UFD over $\mathbbm{k}$ from an affine $\mathbbm{k}$-domain. That topic is also interesting but not what we would mainly discuss in this note. In this case, I will list several important results without proofs. Proofs of those results and relevant definitions can be found in \cite{Sa64}, \cite{Ke05} and \cite{Ha77}.
\end{rem}
\begin{lema}\label{UFDcheckway1}
Given a noetherian domain $A$, following statements are equivalent:\\
(1) $A$ is a UFD;\\
(2) every height one prime ideal of $A$ is principal;\\
(3) $A$ is integrally closed and the Weil divisor class group $\textup{Cl}(\textup{Spec}(A))$ is trivial.
\end{lema}
\begin{rem}
Since each affine $\mathbbm{k}$-domain is noetherian, so Lemma \ref{UFDcheckway1} gives us a way to test affine UFDs. For example, if $V$ is a smooth $\mathbbm{k}$-variety, then from Proposition 6.15 in \cite{Ke05} we know that $\textup{Pic}(\mathbbm{k}[V])\cong\textup{Cl}(\mathbbm{k}[V])$. In this case, $\mathbbm{k}[V]$ is an affine UFD implies that the Picard group $\textup{Pic}(\mathbbm{k}[V])$ should be trivial.
\end{rem}
\begin{exam} This two examples show some power of this lemma.\\
(1) $A_1=\mathbb{R}[x,y]/(x^2+y^2-1)$ is not an affine UFD over $\mathbb{R}$;\\
(2) $A_2=\mathbb{C}[x_1,\dots,x_n]/(x_1^2+\dots+x_n^2-1)$ is not an affine UFD over $\mathbb{C}$ for each $n>2$.
\begin{proof}
In fact, we could obtain that $\textup{Pic}(A_1)=\mathbb{Z}/2\mathbb{Z}$ and $\textup{Pic}(A_2)=\mathbb{Z}$.
\end{proof}
\end{exam}
\begin{ppst}
Given a smooth affine curve $C$ over $\mathbbm{k}$, then its coordinate ring $\mathbbm{k}[C]$ is integrally closed.
\end{ppst}
\begin{exam}
Note that $x^2+y^2=1$ is a smooth affine curve in $\mathbb{A}^2$, the affine $\mathbbm{k}$-domain $A=\mathbb{C}[x,y]/(x^2+y^2-1)$ is an affine UFD over $\mathbb{C}$ since $\textup{Pic(A)}$ is trivial.
\end{exam}
\begin{ppst}\label{UFDcheckway2}
Given an affine $\mathbbm{k}$-domain $A$ and a multiplicatively closed set $S\subseteq A$ generated by a set of prime elements, then the localization $S^{-1}A$ is a UFD implies that $A$ is an affine UFD over $\mathbbm{k}$.
\end{ppst}
\begin{exam}
Given $n>2$, $\mathbb{R}[x_1,\dots,x_n]/(x_1^2+\dots+x_n^2-1)$ is an affine UFD over $\mathbb{R}$.

\begin{proof}
Denote this affine $\mathbb{R}$-domain by $A$, one can see that since
\begin{align*}
A/(x_n-1)&\cong\mathbb{R}[x_1,\dots,x_n]/(x_1^2+\dots+x_{n-1}^2+x_n^2-1,x_n-1)\\
&\cong\mathbb{R}[x_1,\dots,x_n]/(x_1^2+\dots+x_{n-1}^2,x_n-1)\\
&\cong\mathbb{R}[x_1,\dots,x_{n-1}]/(x_1^2+\dots+x_{n-1}^2)
\end{align*}
is an integral domain, $x_{n}-1$ is prime in $A$. Let $t=(x_n-1)^{-1}$, one has
\begin{center}
$\displaystyle x^2_1+\dots+x_n^2-1=0\Rightarrow x_1^2+\dots+x_{n-1}^2+\frac{1}{t^2}+\frac{2}{t}=0$
\end{center}
which implies that $t\in\mathbb{R}[tx_1,\dots,tx_{n-1}]$. Therefore
$$A_{x_n-1}=A[s]/(1-s(z-1))=A[t]=\mathbb{R}[tx_1,\dots,tx_{n-1},\frac{1}{t}]$$
is a UFD and then $A$ is an affine UFD by Proposition \ref{UFDcheckway2}.
\end{proof}
\end{exam}
\begin{rem} What's more, Proposition \ref{UFDcheckway2} is actually a corollary of a famous theorem of Masayoshi Nagata (see Theorem 6.3 in \cite{Sa64}). One can then prove the following well-known theorem given by Masayoshi Nagata and Abraham A. Klein. I fail to find direct sources and give an adapted one according to Theorem 8.2 in \cite{Sa64}.
\end{rem}
\begin{thm}
Given $n\geq5$, if $f(x_1,\dots,x_n)$ is a non-degenerate quadratic form, then affine $\mathbbm{k}$-domain $A=\mathbbm{k}[x_1,\dots,x_n]/(f)$ is an affine UFD.
\end{thm}

\subsection{Nilpotent derivations and Makar-Limanov invariants}
\begin{defn}
Let $A$ be an algebra over $\mathbbm{k}$, if a homomorphism of $\mathbbm{k}$-algebras $D:A\rightarrow A$ satisfies the Leibniz rule
\begin{center}$D(ab)=D(a)b+aD(b)$ for all $a,b\in A$\end{center}
then it is a \textbf{$\mathbbm{k}$-derivation} of $A$. $\textup{Der}_{\mathbbm{k}}(A)$ denotes the set of all $\mathbbm{k}$-derivations of $A$.
\end{defn}
\begin{ppst}\label{k in Der}
Given an algebra $A$ over $\mathbbm{k}$ and a $\mathbbm{k}$-derivation of $A$, then\\
(1) $D(a)=0$ for all $a\in\mathbbm{k}$;\\
(2) $\textup{Ker}(D):=\{a\in A:D(a)=0\}$ is a subalgebra of $A$;\\
(3) $\textup{Nil}(D):=\{a\in A:\exists n_a\in\mathbb{N}_+,D^{n_a}(a)=0\}$ is a subalgebra of $A$.

\begin{proof}
(1) Let $a=b=1$, one has
\begin{center}$D(1)=1D(1)+D(1)1=2D(1)$\end{center}
which implies that $D(1)=0$, then $D(a)=D(a\cdot1)=aD(1)=0$ for all $a\in\mathbbm{k}$.\\
(2) From (1), we know that $\mathbbm{k}\subseteq\textup{Ker}(D)$, so $\textup{Ker}(D)$ can succeed the algebra structure of $A$. Then it is sufficient to prove that $\textup{Ker}(D)$ is a vector space over $\mathbbm{k}$, which is not hard by the definition of vector spaces.\\
(3) From (1), we know that $\mathbbm{k}\subseteq\textup{Nil}(D)$, so $\textup{Nil}(D)$ can succeed the algebra structure of $A$. Then it is sufficient to prove that $\textup{Nil}(D)$ is a vector space over $\mathbbm{k}$, which is not hard by the definition of vector spaces.
\end{proof}
\end{ppst}
\begin{rem}
In (2), one can also prove $\textup{Ker}(D)$ to be a vector space over $\mathbbm{k}$ by the isomorphism theorem for vector spaces (modules over a field). One can also see that $\textup{Ker}(D)\subseteq\textup{Nil}(D)$ as a subalgebra.
\end{rem}
\begin{defn}
Given an algebra $A$ over $\mathbbm{k}$ and a $\mathbbm{k}$-derivation of $A$, if $\textup{Nil}(D)=A$, then $D$ is \textbf{locally nilpotent} or is a \textbf{locally nilpotent $\mathbbm{k}$-derivation}. The set of all locally nilpotent $\mathbbm{k}$-derivations of $A$ is denoted by $\textup{LND}_{\mathbbm{k}}(A)$ or just $\textup{LND}(A)$ when with no hazards to rise confusion.
\end{defn}
\begin{defn}
Given a locally nilpotent $\mathbbm{k}$-derivation over $\mathbbm{k}$-algebra $A$ and a non zero element $a\in A$, we define
\begin{center} $\deg_{D}(a):=\max_{n\in\mathbb{N}}\{D^n(a)\neq0\}$\end{center}
as the \textbf{degree of $a$ by} $D$. Specifically, we define $\deg_{D}(0)=-\infty$.
\end{defn}
\begin{exam}
For each $\mathbbm{k}$-algebra $A$, the zero endomorphism $\theta_A:A\rightarrow A$ which sends all elements to $0$ is a locally nilpotent $\mathbbm{k}$-derivation, called \textbf{zero derivation}.
\end{exam}
\begin{exam}
All partial derivations are locally nilpotent $\mathbbm{k}$-derivation of $\mathbbm{k}^{[n]}$.
\end{exam}
\begin{defn}
Given an algebra $A$ over $\mathbbm{k}$, we define set
\begin{center}$\displaystyle\mathcal{ML}(A):=\bigcap_{D\in\textup{LND}_{\mathbbm{k}}(A)}\textup{Ker}(D) $\end{center}
as \textbf{Makar-Limanov invariant of} $\mathbbm{k}$-algebra $A$.
\end{defn}

\begin{exam}
The partial derivation $\frac{\partial{}}{\partial{x}}:\mathbb{R}[x,y]\rightarrow\mathbb{R}[x,y]$ is a locally derivation over the field of real numbers $\mathbb{R}$ and one can see $\deg_{\frac{\partial{}}{\partial{x}}}(x^3)=3$.
\end{exam}
\begin{exam}
Given a positive integer $n$, $\mathcal{ML}(\mathbbm{k}^{[n]})=\mathbbm{k}$.

\begin{proof}
First, we know that $\mathbbm{k}\subseteq\mathcal{ML}(\mathbbm{k}^{[m]})$ from Proposition \ref{k in Der}. Then since all partial derivations are locally nilpotent derivations over $\mathbbm{k}$, so $\mathbbm{k}\supseteq\mathcal{ML}(\mathbbm{k}^{[m]})$.
\end{proof}
\end{exam}

\begin{exam}\label{example1}
Given an affine UFD $A:=\mathbb{C}[x,y]/(x^2+y^2-1)$ over the complex field $\mathbb{C}$, one has $\mathcal{ML}(A)=A$.

\begin{proof}
It is sufficient to prove that the only element in $\textup{LND}_{\mathbb{C}}(A)$ is zero derivation. Given a nonzero $D\in\textup{LND}(A)$, one has
$$0=D(1)=D(x^2+y^2)=2xD(x)+2yD(y)$$
from which we could get that
\begin{center}$D(x)=yp(x,y)$ and $D(y)=-xp(x,y)$\end{center}
for a nonzero $p(x,y)\in A$. Then one has
$$\deg_D(x)-1=\deg_D(yp(x,y))=\deg_D(y)+\deg_D(p(x,y))$$
but also
$$\deg_D(y)-1=\deg_D(-xp(x,y))=\deg_D(x)+\deg_D(p(x,y))$$
which implies that $\deg_D(p(x,y))=-1$, impossible.
\end{proof}
\end{exam}

\begin{exam} (\cite{ML98}, Example 11)
Given affine domain $A:=\mathbb{C}[x,y]/(x^3-y^2)$ over the complex field $\mathbb{C}$ (by the way, $A$ is not a UFD), then $\mathcal{ML}(A)=A$.

\begin{proof}
It is sufficient to prove that the only element in $\textup{LND}_{\mathbb{C}}(A)$ is zero derivation. If $D\in\textup{LND}(A)$ is not zero, then one has $D(x)\neq c$ for all $c\in\mathbb{C}-\{0\}$; otherwise
\begin{center}$2yD(y)=D(y^2)=D(x^3)=3x^2D(x)=3cx^2$\end{center}
which leads to $D(y)=\frac{3cx^2}{2y}\notin A$, a contradiction.

Therefore, $D(x)$ is nonconstant in $A$. Since $D(y)=\frac{3x^2}{2y}D(x)\in A$, $x|D(x)$ and then one can set $D(x)=yp(x)+xq(x)\in A$. Let $m:=\deg_D(x),n:=\deg_D(y)$, one has
$$n=\frac{1}{2}\deg_D(y^2)=\frac{1}{2}\deg_D(x^3)=\frac{3}{2}m$$
so $n>m$. Notice that
$$0=D^{m+1}(x)=D^{m}(yp(x)+xq(x))=D^{m}(yp(x))+D^{m}(xq(x))$$
and $$D^{m+1}(xq(x))=\frac{\textup{d}^{m+1} xq(x)}{\textup{d}x^k}D^{m+1}(x)=0$$
we could obtain $D^{m+1}(yp(x))=0$. In this case, one can see
$$n+\deg_D(p(x))=\deg_D(y)+\deg_D(p(x))=\deg_D(yp(x))\leq m$$
This is impossible, hence the only element in $\textup{LND}_{\mathbb{C}}(A)$ is zero derivation.
\end{proof}
\end{exam}
\begin{rem}
Personally, I do not really agree with Makar-Limanov's way to compute $\mathcal{ML}(A)$ in this example because he uses the undefined $\deg_D(t)$ where $t$ is a parameter for $A\cong\mathbb{C}[t^2,t^3]$.
\end{rem}

\section{On the Makar-Limanov's Conjecture}
\subsection{Several existed results}
Here we list some results about Conjecture \ref{conj}. From now on, \textbf{transcendence degree} of an affine $\mathbbm{k}$-domain $A$ is denoted by $\textup{tr.deg}_{\mathbbm{k}}A$.

\begin{thm}\label{thm0}\textup{(\cite{ML98}, Lemma 21)}\\
Given an affine domain $A$ over $\mathbbm{k}$. If $\mathcal{ML}(A)=A$, then $\mathcal{ML}(A^{[1]})=A$.
\end{thm}
\begin{rem}
In this case, one can see $\mathcal{ML}(A)=\mathcal{ML}(A^{[1]})=A$ for the affine domain $A$ in Example \ref{example1}.
\end{rem}

\begin{thm}\label{thm1}\textup{(\cite{CM05}, Lemma 2.3)}\\
Given an affine domain $A$ over $\mathbbm{k}$ with $\textup{tr.deg}_{\mathbbm{k}}A=1$. Then\\
(1) $A\cong\mathbbm{k}^{[1]}$ if and only if $\mathcal{ML}(A)=\mathbbm{k}$;\\
(2) $A\neq\mathbbm{k}^{[1]}$ if and only if $\mathcal{ML}(A)=A$.
\end{thm}
\begin{rem}
Here we could get that $\mathcal{ML}(A)=\mathcal{ML}(A^{[1]})$ holds for each affine domain $A$ over $\mathbbm{k}$ with $\textup{tr.deg}_{\mathbbm{k}}A=1$.
\end{rem}
\begin{thm}\label{thm2}\textup{(\cite{Fr17}, Theorem 9.12)}\\
Given an affine unique factorization domain $A$ over an algebraically closed field $\mathbbm{k}$ with $\textup{tr.deg}_{\mathbbm{k}}A=2$. Then $\mathcal{ML}(A)=\mathbbm{k}$ is equivalent to $A=\mathbbm{k}^{[2]}$.
\end{thm}

\begin{rem}
Example \ref{example2} shows that $\textup{tr.deg}_{\mathbbm{k}}A=2$ is necessary for this theorem.
\end{rem}

\begin{exam}\label{example2}
Consider the affine $\mathbb{C}$-domain $A=\mathbb{C}[x,y,z,w]/(xz-yw-1)$, one has $A$ is an affine UFD over $\mathbb{C}$ and $\mathcal{ML}(A)=\mathbb{C}$.

\begin{proof}
(1) First, since
$$A/(w)\cong\mathbb{C}[x,y,z]/(xy-1)=(\mathbb{C}[x,y]/(xy-1))[z]\cong(\mathbb{C}[x][\frac{1}{x}])[z]\cong(\mathbb{C}[x])_x[z]$$
is an integral domain (actually an affine UFD over $\mathbb{C}$), $w$ is prime in $A$.\\
Let $x'=\frac{x}{w},y'=y,z'=z$, then $x',y',z'$ are algebraically independent and
$$A_w\cong A[\frac{1}{w}]=\mathbb{C}[x',y',z'][\frac{1}{w}]=\mathbb{C}[x',y',z']$$
is a UFD over $\mathbb{C}$. Therefore, by Proposition \ref{UFDcheckway2}, $A$ is an affine UFD over $\mathbb{C}$.\\
(2) One can check that the derivations $D_1,D_2,D_3$ and $D_4$ on $A$ generated by
$$\begin{matrix}
D_1(x)=0, &D_1(y)=z, &D_1(z)=0, &D_1(w)=x \\
D_2(x)=0, &D_2(y)=w, &D_2(z)=x, &D_2(w)=0 \\
D_3(x)=z, &D_3(y)=0, &D_3(z)=0, &D_3(w)=y \\
D_4(x)=w, &D_4(y)=0, &D_4(z)=y, &D_4(w)=0
\end{matrix}$$
are locally nilpotent. Here one can see $\textup{Ker}(D_1)$ and $\mathbb{C}[x,z]$ are both algebraically closed in $A$. What's more, one has $\mathbb{C}[x,z]\subseteq\textup{Ker}(D_1)\subset A$ and $$\textup{tr.deg}_{\mathbb{C}}\textup{Ker}(D_1)=\textup{tr.deg}_{\mathbb{C}}\mathbb{C}[x,z]=2$$
Hence $\textup{Ker}(D_1)=\mathbb{C}[x,z]$. Similarly, we could obtain $\textup{Ker}(D_2)=\mathbb{C}[x,w],\textup{Ker}(D_3)=\mathbb{C}[y,z]$  and $\textup{Ker}(D_4)=\mathbb{C}[y,w]$. In this case,
$$\mathbb{C}=\bigcap^4_{i=1}\textup{Ker}(D_i)\supseteq\mathcal{ML}(A)\supseteq\mathbb{C}$$
and thus $\mathcal{ML}(A)=\mathbb{C}$. Moreover, since $\mathbb{C}\subseteq\mathcal{ML}(A^{[1]})\subseteq\mathcal{ML}(A)=\mathbb{C}$, it is clear that $\mathcal{ML}(A^{[1]})=\mathcal{ML}(A)=\mathbb{C}$.
\end{proof}
\end{exam}

\subsection{On Danielewski domains}
\begin{defn}
Given integer $n\geq0$ and a polynomial $p\in\mathcal{S}$ where
$$\mathcal{S}:=\bigcup_{d\geq1}\{p\in\mathbbm{k}[x,y]|d=\deg_y{p}=\deg_{y}p(0,y)\}$$
we define the subring $\mathbbm{D}_{(n,p)}$ of $\mathbbm{k}[x,x^{-1},y]$ by
$$\mathbbm{D}_{(n,p)}:=\mathbbm{k}[x,y,z]/(x^nz-p(x,y))$$
to be a \textbf{Danielewski domain} over $\mathbbm{k}$.
\end{defn}
\begin{rem}\label{Daniel0}
Here we should notice that we could assume $\deg_xp(x,y)<n$. If a monomial $x^sy^t$ appears in $p$ with $s\geq n$, then one has $x^n(z-x^{s-n}y^t)=p'(x,y)$ for $p'(x,y)=p(x,y)-x^sy^t$. Therefore, one has $\mathbbm{D}_{(n,p)}\cong\mathbbm{D}_{(n,p')}$. Repeating this, we could eventually obtain a $p_0(x,y)\in\mathcal{S}$ with $\deg_xp_0(x,y)<n$ and $\mathbbm{D}_{(n,p)}\cong\mathbbm{D}_{(n,p_0)}$.
\end{rem}
Danielewski domains over $\mathbbm{k}$ are clearly affine $\mathbbm{k}$-domains and we want to know when one such domain would become an affine UFD over $\mathbbm{k}$. Firstly, we introduce the following simple and well-known result.
\begin{ppst}\label{Daniel1}
The Danielewski domain
$$\mathbbm{D}_{(1,p)}=\mathbbm{k}[x,y,z]/(xz-p(y))$$
is an affine UFD if and only if $p\in\mathbbm{k}[y]-\mathbbm{k}$ is irreducible.

\begin{proof}
If nonconstant polynomial $p(y)\in\mathbbm{k}[y]$ is irreducible, then the ideal $(p(y))$ is prime in both $\mathbbm{k}[y]$ and $\mathbbm{k}[y,z]$. Therefore, the quotient ring
$$\mathbbm{D}_{(1,p)}/(x)=\mathbbm{k}[y,z]/(p(y))$$
is an integral domain which implies that $x$ is prime in $\mathbbm{D}_{(1,p)}$. Note that
$$(\mathbbm{D}_{(1,p)})_x\cong\mathbbm{D}_{(1,p)}[\frac{1}{x}]=\mathbbm{k}[x,y,\frac{1}{x}]$$
is a UFD over $\mathbbm{k}$, $\mathbbm{D}_{(1,p)}$ is an affine UFD over $\mathbbm{k}$ according to Proposition \ref{UFDcheckway2}.\\
On the other hand, a reducible $p\in\mathbbm{k}[y]-\mathbbm{k}$ can be decomposed as a product of nonconstant polynomials in $\mathbbm{k}$. In this case, $\mathbbm{D}_{(1,p)}$ cannot be UFD or each irreducible element in $\mathbbm{D}_{(1,p)}$ is prime, which implies a contradiction form $xz=p(y)$.
\end{proof}
\end{ppst}

\begin{exam}\label{example3}
Consider affine $\mathbbm{k}$-domain $A=\mathbbm{k}[x,y,z]/(xz-y^2-1)$, one can see from Proposition \ref{Daniel0} that $A$ is an affine UFD when $\mathbbm{k}=\mathbb{R}$ and is not when $\mathbbm{k}=\mathbb{C}$.
\end{exam}

\begin{ppst}\label{example3+}
If Danielewski domain $\mathbbm{D}_{(1,p)}$ is a UFD, then $\mathcal{ML}(\mathbbm{D}_{(1,p)})=\mathbbm{k}$.

\begin{proof}
One can check that the derivations $D_1$ and $D_2$ on $\mathbbm{D}_{(1,p)}$ generated by
$$\begin{matrix}
D_1(x)=0, &D_1(y)=x, &D_1(z)=p'(y) \\
D_2(x)=p'(y), &D_2(y)=z, &D_2(z)=0
\end{matrix}$$
are locally nilpotent. Here one can see $\textup{Ker}(D_1)$ and $\mathbbm{k}[x]$ are both algebraically closed in $\mathbbm{D}_{(1,p)}$. What's more, one has $\mathbbm{k}[x]\subseteq\textup{Ker}(D_1)\subset\mathbbm{D}_{(1,p)}$ and $$\textup{tr.deg}_{\mathbbm{k}}\textup{Ker}(D_1)=\textup{tr.deg}_{\mathbbm{k}}\mathbbm{k}[x]=1$$
Hence $\textup{Ker}(D_1)=\mathbbm{k}[x]$. Also, we could obtain that $\textup{Ker}(D_2)=\mathbbm{k}[z]$. In this case,
$$\mathbbm{k}=\textup{Ker}(D_1)\cap\textup{Ker}(D_2)\supseteq\mathcal{ML}(\mathbbm{D}_{(1,p)})\supseteq\mathbbm{k}$$
and thus $\mathcal{ML}(\mathbbm{D}_{(1,p)})=\mathbbm{k}$, which implies that $\mathcal{ML}((\mathbbm{D}_{(1,p)})^{[1]})=\mathbbm{k}$.
\end{proof}
\end{ppst}

This proposition implies that Makar-Limanov's conjecture holds for $\mathbbm{D}_{(1,p)}$ and we want to look into other cases. Furthermore, Alhajjar proved the following theorem in his Ph.D. thesis \cite{Al15} based on some works of Makar-Limanov and Freudenburg listed this result in his book \cite{Fr17} (Theorem 9.2).

\begin{thm} \textup{(\cite{Al15}, Proposition 6.16)}
If $n,d\geq2$ and $\deg_y{p(x,y)}=d$ for $p\in\mathcal{S}$, then $\mathcal{ML}(\mathbbm{D}_{(n,p)})=\mathbbm{k}[x]$.
\end{thm}

Note that if $\mathbbm{D}_{(n,p)}$ is an affine UFD with $n,d\geq2$ and $\deg_y{p(x,y)}=d$, then $x$ should be prime and the quotient ring
$$\mathbbm{D}_{(n,p)}/(x)=\mathbbm{k}[y,z]/(p_y(0,y))$$
is an integral domain. Therefore, $p(0,y)\in\mathbbm{k}[y]$ is irreducible. However, it is impossible when $\mathbbm{k}$ is algebraically closed because $\deg_yp(0,y)=\deg_yp(x,y)=d>1$. So here we only need to consider potential counterexamples to Makar-Limanov's conjecture with a non-algebraically closed field.

\begin{exam}
Consider the Danielewski domain $A=\mathbb{R}[x,y,z]/(x^2z-(y^2+1))$, we could prove the following statements.\\
(1) $A$ is an affine UFD;\\
(2) $\mathcal{ML}(A^{[1]})=\mathbb{R}$.

\begin{proof}
(1) First, since $A/(x)=\mathbb{R}[y,z]/(y^2+1)$ is an integral domain, $x$ is prime in affine $\mathbb{R}$-domain $A$. Then from $z=\frac{y^2+1}{x^2}$, one can see
$$A_x\cong A[\frac{1}{x}]=\mathbb{R}[x,y,\frac{1}{x}]$$
is a UFD. Therefore, $A$ is an affine UFD by Proposition \ref{UFDcheckway2}.\\
(2) Consider the Danielewski domain $B=\mathbb{R}[x,y,z]/(xz-(y^2+1))$, we know that it is an affine UFD by Example \ref{example3} and then
$$\mathcal{ML}(B^{[1]})=\mathcal{ML}(B)=\mathbb{R}$$
from Proposition \ref{example3+}. Moreover, according to Theorem 10.1 in \cite{Fr17}, one has
$$\mathcal{ML}(A^{[1]})=\mathcal{ML}(B^{[1]})=\mathbb{R}$$
so we are done.
\end{proof}
\end{exam}
\begin{rem}
This example implies that when $\mathbbm{k}$ is not algebraically closed, it is possible that $\mathcal{ML}(A^{[1]})\neq\mathcal{ML}(A)$ for certain affine UFD over $\mathbbm{k}$. Therefore, we will only consider the conjecture on \emph{algebraically closed fields} in the following part.
\end{rem}

Now we are moving to Danielewski domains $\mathbbm{D}_{(n,p)}$ with $\deg_y p(x,y) =1$ and $n\geq2$. In this case, the polynomial $p(x,y)=a(x)y+b(x)$ where $a,b\in\mathbbm{k}[x]$ and $a(0)\neq0$. Note that $\gcd{(x^n,a(x))}=1$, then one can see
$$\mathbbm{k}[x,y,z]/(x^nz-a(x)y-b(x))\cong\mathbbm{k}[x,y,z]/(z)\cong\mathbbm{k}[x,y]$$
from $\mathbbm{k}[x,y,z]\cong\mathbbm{k}[x,y,x^nz-a(x)y-b(x)]$. Therefore, the Makar-Limanov invariant of one such Danielewski domain is $\mathbbm{k}$.

\vspace{2mm}
Based on the previous discussions in this part, we clearly have the following result as a conclusion.

\begin{thm}
If $\mathbbm{k}$ is algebraically closed, then $\mathcal{ML}(\mathbbm{D}_{(n,p)})=\mathcal{ML}((\mathbbm{D}_{(n,p)})^{[1]})$ for each Danielewski domain $\mathbbm{D}_{(n,p)}$ over $\mathbbm{k}$.
\end{thm}

\newpage
\subsection{On Koras-Russell threefolds}
\begin{defn}
A \textbf{Koras-Russell threefold of the first kind} is defined by
$$R=\mathbbm{k}[x,y,z,w]/(x+x^dy+z^u+w^v)$$
where $d,u,v\geq2$ and $\gcd(u,v)=1$.
\end{defn}

In this part, we are going to investigate the conjecture on Koras-Russell threefolds which are clearly affine domains. First of all, we want to know when a Koras-Russell threefold $R$ is an affine UFD. Note that
$$x+x^dy+z^u+w^v=0\Rightarrow y=-\frac{1}{x^{d-1}}-\frac{z^u}{x^d}-\frac{w^v}{x^d}\in\mathbbm{k}[x,z,w][\frac{1}{x}]$$
then one can see
$$R_x\cong R[\frac{1}{x}]=\mathbbm{k}[x,z,w][\frac{1}{x}]$$
Hence $R_x$ is a localization of a polynomial ring which is a UFD. It is not hard to check that the element $x$ is prime in $R$. Therefore, we are able to see the following result according to Proposition \ref{UFDcheckway2}.

\begin{ppst}
Each Koras-Russell threefold $R$ is an affine UFD.
\end{ppst}

Then we wonder if $\mathcal{ML}(R)=\mathcal{ML}(R^{[1]})$ also always holds. In order to figure out this problem, we need to introduce a fundamental result at first.

As a concise version of Makar-Limanov's original proof in \cite{ML96}, Daigle, Freudenburg and Moser-Jauslin proved a feature of Koras-Russell threefolds in \cite{DFM17}. Based on their result, Freudenburg computed Makar-Limanov invariants of those threefolds in new edition of his book \cite{Fr17}.

\begin{thm}\label{KR0} \textup{(\cite{Fr17}, Theorem 9.9)}
Let $R$ be a Koras-Russell threefold
$$R=\mathbbm{k}[x,y,z,w]/(x+x^dy+z^u+w^v)$$
where $d,u,v\geq2$ and $\gcd(u,v)=1$. Then $\mathcal{ML}(R)=\mathbbm{k}[x]$.
\end{thm}

In their arguments and computations, some propositions and lemmas are important. We are going to introduce some less-known definitions relevant to those propositions and lemmas at first and then give the propositions and lemmas.

\begin{defn}
Let $G$ be an abelian group, and let $B$ be a ring. A \textbf{$G$-grading} of $B$ is a family $\{B_a\}_{a\in G}$ of subgroups of $(B,+)$ such that
\begin{center}
$B=\oplus_{a\in G}B_a$ and $B_aB_b\subseteq B_{a+b}$ for all $a,b\in G$
\end{center}
and a ring (or domain/integral domain/affine $\mathbbm{k}$-domain) $B$ with a $G$-grading is called a \textbf{$G$-grading ring} (or domain/integral domain/affine $\mathbbm{k}$-domain).
\end{defn}

\begin{defn}
Given $B=\oplus_{a\in G}B_a$ a $G$-grading ring, a nonzero $f\in B$ is called \textbf{$G$-homogeneous} if $f\in B_a$ for a unique $a\in G$. Here we say that $f$ is of degree $a$ and write $\deg_Gf=a$.
\end{defn}

\begin{defn}
Given $B=\oplus_{a\in G}B_a$ a $G$-grading ring and a nonzero $f\in B$. $\overline{f}$ denotes the \textbf{highest-degree homogeneous summand} of $f$. In case $f=0$, we define $\overline{0}=0$. Here one can see $\overline{x}=x$ if $x$ is $G$-homogeneous.
\end{defn}

\begin{defn}
Given $B=\oplus_{a\in G}B_a$ a $G$-grading $\mathbbm{k}$-domain and a $\mathbbm{k}$-algebra $R\subseteq B$, we define the \textbf{associated $G$-graded domain} $\overline{R}$ to be the $\mathbbm{k}$-subalgebra of $B$ generated by the set $\{\overline{r}|r\in R,r\neq0\}$.
\end{defn}

\begin{defn}
Given $B=\oplus_{a\in G}B_a$ a $G$-grading $\mathbbm{k}$-domain and $D\in\textup{Der}_{\mathbbm{k}}B$, if
$$\deg_GD:=\max\{\deg_G(Df)-\deg_G(f):f\neq0,f\in B\}$$
exists, then we define $\overline{D}:\overline{B}\rightarrow\overline{B}$ as
$$\overline{D}\overline{x}=
\left\{\begin{aligned}
\overline{Dx},&\quad \deg_G(Df)-\deg_G(f)=\deg_GD\\
0,&\quad \deg_G(Df)-\deg_G(f)<\deg_GD
\end{aligned}\right.
$$
Moreover, one can see $\deg_G(\overline{D})=\deg_GD$ and $\textup{Ker}(D)\subseteq\textup{Ker}(\overline{D})$.
\end{defn}

\begin{lema}\label{KR1} \textup{(\cite{DFM17}, Lemma 3.7)}
Let $(G,\leq,+,0)$ be a totally ordered abelian group and $B=\oplus_{a\in G}B_a$ a $G$-graded integral domain. Let $A=\oplus_{a\leq0}B_a,x\in B$, and $R=A[x]$. Then $\overline{R}=A[\overline{x}]$.
\end{lema}

\begin{ppst}\label{KR2} \textup{(\cite{DFM17}, Theorem 3.8)}
Let $(G,\leq,+,0)$ be a totally ordered abelian group, $B$ a $G$-graded $\mathbbm{k}$-affine domain, and $R\subset B$ a $\mathbbm{k}$-subalgebra such that $B$ is a localization of $R$. Let $\deg_G:R\rightarrow G\cup\{-\infty\}$ be the restriction of the degree function on $B$ determined by the grading. Then $\deg(D)$ is defined for every $D\in\textup{Der}_{\mathbbm{k}}(R)$.
\end{ppst}

\begin{ppst}\label{KR3} \textup{(\cite{DFM17}, Corollary 6.3)}
Let $(G,\leq,+,0)$ be a totally ordered abelian group, $B=\oplus_{a\in G}B_a$ a $G$-graded integral domain containing $\mathbb{Z}$, where $B$ is finitely generated as a $A$-algebra. Then one can write $B=A[x_1,\dots,x_n]$ where $x_i\neq0$ is homogeneous of degree $d_i\neq0$ for each $i$. Let $H_i=\langle d_1,\dots,\hat{d}_i,\dots,d_n\rangle$ for each $i\in[n]$. Then for every $G$-homogeneous $D\in\textup{LND}_{\mathbbm{k}}(B)$ the following conditions hold.\\
(1) For each $i\in [n]$ such that $H_i\neq G(B)$, $D^2x_i=0$.\\
(2) For every choice of distinct $i,j\in [n]$ such that $H_i\neq G(B)$ and $H_j\neq G(B)$, one has $Dx_i$ or $Dx_j=0$.
\end{ppst}

In line with their works, we are able to make the following computation.

\begin{thm}
Let $R$ be a Koras-Russell threefold
$$R=\mathbbm{k}[x,y,z,w]/(x+x^dy+z^u+w^v)$$
where $d,u,v\geq2$ and $\gcd(u,v)=1$. Then $\mathcal{ML}(R[t])=\mathbbm{k}[x]$.
\end{thm}

\begin{proof}
Let group $G=\mathbb{Z}^3$ and define a total order $\preceq$ on $G$ by lexicographical ordering. Consider a $G$-grading on $B=\mathbbm{k}[x,x^{-1},z,w,t]$ with $x,z,w,t$ homogeneous and
$$\deg_G(x,z,w,t)=((-1,0,0),(0,-v,0),(0,-u,0),(0,0,-1))$$
Note that $\gcd(u,v)=1$, one has
$$\{f\in B|\deg(f)\preceq(0,0,0)\}=\mathbbm{k}[x,z,w,t]\subset R[t]$$
We set $A=\{f\in B|\deg(f)\preceq(0,0,0)\}=\mathbbm{k}[x,z,w,t]$. The degree function $\deg_G$ on $B$ restricts to affine $\mathbbm{k}$-domain $R[t]$, where $\deg_G{y}=(d,-uv,0)$. According to Lemma \ref{KR1}, one has $\overline{R[t]}=A[\overline{y}]=\mathbbm{k}[x,z,w,t,\overline{y}]$. Since $y=-x^{-d}(x+z^u+w^v)$ in $B$, one can see that $\overline{y}=-x^{-d}(z^u+w^v)$ and then $x^d\overline{y}+z^u+w^v=0$.\\
Given a nonzero $D\in\textup{LND}_{\mathbbm{k}}(R[t])$. By Proposition \ref{KR2}, the induced $G$-homogeneous derivation $\overline{D}$ of $\overline{R[t]}$ is nonzero and locally nilpotent. Since
$$\langle \deg_G{\overline{y}},\deg_G{z},\deg_G{w},\deg_G{t} \rangle=d\mathbb{Z}\times\mathbb{Z}\times\mathbb{Z}$$
$$\langle \deg_G{x},\deg_G{\overline{y}},\deg_G{w},\deg_G{t} \rangle=\mathbb{Z}\times u\mathbb{Z}\times\mathbb{Z}$$
$$\langle \deg_G{x},\deg_G{\overline{y}},\deg_G{z},\deg_G{t} \rangle=\mathbb{Z}\times v\mathbb{Z}\times \mathbb{Z}$$
are proper subgroups of $G(\overline{R[t]})=\mathbb{Z}^3$, we know that at least two of $\overline{D}x,\overline{D}z,\overline{D}w$ must be zero according to Proposition \ref{KR3}.\\
When $\overline{D}z=\overline{D}w=0$, one has $\overline{D}(x^d\overline{y})=0$ and then $\overline{D}x=\overline{D}\overline{y}=0$. In this case, the only possible is $\overline{D}t\neq0$.\\
When $\overline{D}x=\overline{D}z=0$ or $\overline{D}x=\overline{D}z=0$, one can see that either $\overline{D}\overline{y}\neq0$ or
$$\overline{D}x=\overline{D}\overline{y}=\overline{D}z=\overline{D}w=0,\overline{D}t\neq0$$
Therefore, either $\textup{Ker}\overline{D}\subset A$ or $\textup{Ker}\overline{D}=\mathbbm{k}[x,\overline{y},z,w]$.\\
Now suppose that $Dx\neq0$. Choose $f,g\in\textup{Ker}D$ which are algebraically independent. Let $f_1,g_1\in\mathbbm{k}[x,y,z,w,t]$ and $f_2,g_2\in\mathbbm{k}[y,z,w,t]$ be such that
\begin{center}
$f=xf_1+f_2$ and $g=xg_1+g_2$
\end{center}
Here $f_2$ and $g_2$ should be algebraically independent in $R$. Otherwise, there exists $P\in\mathbbm{k}^{[2]}$ with $P(f_2,g_2)=0$. But then $P(f,g)\in xR[t]$, which implies that $Dx=0$, a contradiction. In addition, since $\deg_G(xf_1)\prec\deg_G f_2$, one has $\deg_Gf=\deg_G f_2$. Similarity, $\deg_Gg=\deg_Gg_2$. Now one has $\overline{f},\overline{g}\in\textup{Ker}\overline{D}$ where
\begin{center}
$\overline{f}=\overline{f}_2=f_2(\overline{y},z,w,t)$ and $\overline{g}=\overline{g}_2=g_2(\overline{y},z,w,t)$
\end{center}
If $\textup{Ker}\overline{D}\subset A$, then $\overline{D}\overline{y}\neq0$ and $f_2,g_2\in\mathbbm{k}[z,w,t]$. In this case, $\mathbbm{k}[z,w,t]$ is the algebraic closure of $\mathbbm{k}[\overline{f},\overline{g}]$ and thus $\mathbbm{k}[z,w,t]\subset\textup{Ker}\overline{D}$. However, we could obtain $0=\overline{D}(x^d\overline{y})$ and then $\overline{D}=0$, a contradiction.\\
If $\textup{Ker}\overline{D}=\mathbbm{k}[x,\overline{y},z,w]$, then $\overline{D}$ restricts to be a zero derivation in $\overline{R}$. So $D$ restricts to a zero derivation in $R$, a contradiction.\\
Therefore, we always get a contradiction with $Dx\neq0$, which implies that $Dx=0$. So one has $\mathbbm{k}[x]\subseteq\mathcal{ML}(R[t])$. From Theorem \ref{KR0}, we know that
$$\mathcal{ML}(R[t])\subseteq\mathcal{ML}(R)=\mathbbm{k}[x]$$
and thus $\mathcal{ML}(R[t])=\mathbbm{k}[x]$.
\end{proof}

Here one has proved that the Conjecture \ref{conj} holds for all Koras-Russell threefolds. However, methods applied in this proof rely on specific structures of Koras-Russell threefolds and are hard to be generalized.

\subsection{On Finston-Maubach domains}

In paper \cite{FM10}, Finston and Maubach construct a series of affine UFDs (called Finston-Maubach domains in this paper) with non-trivial Makar-Limanov invariant based on Brieskorn-Catalan-Fermat rings. We are going to check Conjecture \ref{conj} on Finston-Maubach domains in the following part.

\begin{defn}
Given $n\geq 3$, we define $F\in\mathbb{C}[x_1,\dots,x_n,y_1,\dots,y_n]$ as
$$F:=x_1^{d_1}+\cdots+x_n^{d_n}+(x_2y_1-x_1y_2)^{e_2}+\cdots+(x_ny_1-x_1y_n)^{e_n}$$
where
$$\frac{1}{d_1}+\cdots+\frac{1}{d_n}+\frac{1}{e_2}+\dots+\frac{1}{e_n}\leq\frac{1}{2n-3}$$
Then the affine domain $R$ given by
$$R:=\mathbb{C}[x_1,\dots,x_n,y_1,\dots,y_n]/(F)$$
is called a \textbf{Finston-Maubach domain} of order $n$.
\end{defn}
\begin{ppst}
Each Finston-Maubach domain $R$ is an affine UFD.

\begin{proof}
It is clear that $R$ is an affine domain. Moreover, Lemma 3.5 in paper \cite{FM10} tells us that $R$ is a UFD. One can also try to prove it independently by Proposition \ref{UFDcheckway2} and the factoriality of Brieskorn-Catalan-Fermat rings with $n\geq5$.
\end{proof}
\end{ppst}

Then we are going to compute the Makar-Limanov invariant of Finston-Maubach domains based on Finston and Maubach's results. For that, one has to introduce a lemma at first.

\begin{lema}\label{FMlemma1} \textup{(\cite{FM10}, Lemma 2.7)}
Let $A$ be an affine $\mathbb{Q}$-domain. Consider a subset $\mathcal{F}=\{f_1,f_2,\dots,f_n\}$ of $A$ and positive integers $d_1,\dots,d_n$ satisfying:
\begin{itemize}
  \item $f:=f_1^{d_1}+f_2^{d_2}+\cdots+f_n^{d_n}$ is a prime element of $A$.
  \item Non nontrivial sub-sum of $F_1^{d_1},F_2^{d_2},\dots,F_n^{d_n}$ lies in $(f)$.
\end{itemize}
Additionally, assume that
$$\frac{1}{d_1}+\frac{1}{d_2}+\cdots+\frac{1}{d_n}\leq\frac{1}{n-2}$$
Set $R:=A/(f)$ and let $D\in\textup{LND}(R)$, one has $D(f_i)=0$ for all $1\leq i\leq n$.
\end{lema}

\begin{ppst}
$\mathcal{ML}(R)=\mathbb{C}^{[n]}$ for a Finston-Maubach domain $R$ of order $n$.

\begin{proof}
Firstly, it is not hard to check that $D_0\in\textup{LND}(R)$ where $D_0$ is generated by $D_0(x_i)=0$ and $D_0(y_i)=x_i$. Given $D\in\textup{LND}(R)$, by Lemma \ref{FMlemma1}, one has $D(x_i)=0$ and $D(l_i)=0$ where $l_i=x_iy_1-x_1y_i$. Therefore, one has
$$x_1D(y_i)=y_1D(x_i),\quad\forall i\in\{2,3,\dots,n\}$$
Since $R$ is a UFD, $D(y_i)=\alpha x_i$ for some $\alpha\in R$. Then one can see $D=\alpha D_0$. Notice that $D$ is nonzero if and only if $\alpha$ is nonzero, one can see $\textup{Ker}(D)=\mathbb{C}[x_1,\dots,x_n]$ for each nonzero $D\in\textup{LND}(R)$. Hence one has $\mathcal{ML}(R)=\mathbb{C}[x_1,\dots,x_n]=\mathbb{C}^{[n]}$.
\end{proof}
\end{ppst}

Now we are going to compute $\mathcal{ML}(R^{[1]})$ for a Finston-Maubach domain $R$ of order $n$ and we also need to introduce several lemmas in \cite{Bon09} and \cite{FM10}.

\begin{lema}\label{FMlemma2} \textup{(\cite{Bon09}, Theorem 3.1)}
Let $f_1,f_2,\dots,f_n\in K[s]$, where $K$ is an algebraically closed field containing $\mathbb{Q}$, Assume $f_1^{d_1}+f_2^{d_2}+\cdots+f_n^{d_n}=0$. Additionally, assume that for every $1\leq i_1<i_2<\cdots<i_s\leq n$,
$$f_{i_1}^{d_{i_1}}+f_{i_2}^{d_{i_2}}+\cdots+f_{i_s}^{d_{i_s}}=0\Longrightarrow\textup{gcd}\{f_{i_1},f_{i_2},\cdots,f_{i_s}\}=1$$
Then
$$\sum_{i=1}^n\frac{1}{d_i}\leq\frac{1}{n-2}$$
implies that all $f_i$ are constant.
\end{lema}

\begin{lema}\label{FMlemma3} \textup{(\cite{FM10}, Lemma 2.2)}
Let $D$ be a locally nilpotent derivation on a domain $A$ containing $\mathbb{Q}$. Then $A$ embeds into $K[s]$, where $K$ is some algebraically closed field of characteristic zero, in such way $D=\partial_s$ on $K[s]$.
\end{lema}

\begin{ppst}\label{FMproof}
$\mathcal{ML}(R^{[1]})=\mathbb{C}^{[n]}$ for a Finston-Maubach domain $R$ of order $n$.

\begin{proof}
Let Finston-Maubach domain $R=\mathbb{C}[x_1,\dots,x_n,y_1,\dots,y_n]/(F)$ where
$$F:=x_1^{d_1}+\cdots+x_n^{d_n}+(x_2y_1-x_1y_2)^{e_2}+\cdots+(x_ny_1-x_1y_n)^{e_n}$$
Suppose $D\in\textup{LND}(R[t])$, where $D\neq0$. By Lemma \ref{FMlemma3} with $K$ an algebraic closure of the quotient field of $(R[t])^D$, we realize $D$ as partial derivation $\partial_s$ on $K[s]\supseteq R[t]$. Notice that $F(s)=0$ and there cannot be a subsum of $x_1^{d_1}+\cdots+x_n^{d_n}+l_2^{e_2}+\cdots+l_n^{e_n}$ to be zero, where $l_i=x_iy_1-x_1y_i$. Then by Lemma \ref{FMlemma2}, one has $x_i,l_i$ to be constant in $K[s]$. In this case, one has $D(x_i)=\partial_s x_i=0$ and $D(y_i)=\partial_s l_i=0$ in $R[t]$. Therefore, one has $\mathcal{ML}(R[t])\supseteq\mathbb{C}[x_1,\dots,x_n]$. In other hands, we already know that $$\mathcal{ML}(R[t])\subseteq\mathcal{ML}(R)=\mathbb{C}[x_1,\dots,x_n]$$
So $\mathcal{ML}(R^{[1]})=\mathcal{ML}(R)=\mathbb{C}^{[n]}$.
\end{proof}
\end{ppst}

Here we know that Conjecture \ref{conj} is also valid as to Finston-Maubach domains. The proof in Proposition \ref{FMproof} is nothing new but to imitate the proof of Finston and Maubach for the almost rigidity of Finston-Maubach domains.

\section{Comments}
This work was originally started as an extension of my Bachelor thesis. At that time, I planned to study Conjecture \ref{conj} by observing examples. However, I found that most of existing methods of computing Makar-Limanov invariants rely on specific structure of an affine domain. So it may be hard to obtain a promising way to solve the conjecture by them. Additionally, my first semester for a master degree will begin soon and I will hardly have time for this work. In this case, I decide to stop it here.

\vspace{2mm}
In my opinion, the core of this conjecture is to build a connection between the UFD property and Makar-Limanov invariant. It is difficult because the UFD itself is a really profound area, which is too fundamental to be given a description by locally nilpotent derivations. If we want to use current ways to compute Makar-Limanov invariants, we need to focus on the structure of $(f)$ in an affine UFD $\mathbbm{k}^{[n]}/(f)$. Or one can try to interpret Makar-Limanov invariants from the perspective of geometry, which may bring us some new ideas.

\vspace{2mm}
These three examples in this paper are computed by three different methods. The first one is simple and direct. The second one uses a famous technique in this area, called homo-generalization. Readers can find more information about this technique in \cite{Fr17}, \cite{KM00} and \cite{ML98}. The third example has some connection with the famous ABC-theorem and one can find more details in \cite{Bon09} and \cite{Fr17}.


\newpage
\begin{thebibliography}{ABCD}

\bibitem[1]{Al15} Bachar Alhajjar. \emph{Locally Nilpotent Derivations of Integral Domains}. Ph.D. thesis, Universite de Bourgogne, 2017.
\bibitem[2]{Bon09} Michiel de Bondt. Another generalization of Mason's ABC-theorem. Preprint, \href{https://arxiv.org/pdf/0707.0434.pdf}{arXiv:0707.0434}, 2009.
\bibitem[3]{CM05} Anthony Crachiola and Leonid Makar-Limanov. On the rigidity of small domains. \emph{Jounral of Algebra}, 284(1):1-12, 2005.
\bibitem[4]{CM09} Anthony Crachiola and Leonid Makar-Limanov. An algebraic proof of a cancellation theorem for surfaces. \emph{Jounral of Algebra}, 320(8):3113-3119, 2008.
\bibitem[5]{DFM17} Daniel Daigle, Gene Freudenburg and Lucy Moser-Jauslin. Locally nilpotent derivations of rings graded by an abelian group. \emph{Advanced Studies in Pure Mathematics}, 75:29-48, 2017.
\bibitem[6]{FM10} David Finston and Stefan Maubach. Constructing (almost) rigid rings and a UFD having infinitely generated Derksen and Makar-Limanov invariant. \emph{Canadian Mathematics Bulletin}, 53(1):77-86, 2010.
\bibitem[7]{Fr17} Gene Freudenburg. \emph{Algebraic Theory of Locally Nilpotent Derivations (Second Edition)}. Springer, New York, The United States, 2017.
\bibitem[8]{Ha77} Robin Hartshorne. \emph{Algebraic Geometry}. Springer, New York, The United States, 1977.
\bibitem[9]{Ke05} Gregor Kemper. \emph{A Course in Commutative Algebra}. Springer, Berlin, Germany, 2009.
\bibitem[10]{KM97} Shulim Kaliman and Leonid Makar-Limanov. On the Russell-Koras contractible threefolds. \emph{Jounral of Algebraic Geometry}, 6(2):247-268, 1997.
\bibitem[11]{KM00} Shulim Kaliman and Leonid Makar-Limanov. AK-invariant of affine domains. In \emph{Affine algebraic geometry}, pages 231-255. Osaka University Press, Osaka, Japan, 2007.
\bibitem[12]{ML98} Leonid Makar-Limanov. \emph{Locally nilpotent derivations, a new ring invariant and applications}. Lecture notes, \href{https://www.researchgate.net/publication/265356937_Locally_nilpotent_derivations_a_new_ring_invariant_and_applications}{ResearchGate:265356937}, 1998.
\bibitem[13]{ML96} Leonid Makar-Limanov. On the hypersurface $x+x^2y+z^2+t^3$ in $\mathbb{C}^4$ or a $\mathbb{C}^3$-like threefold which is not $\mathbb{C}^3$. \emph{Israel Journal of Mathematics}, 96:419-429, 1996.
\bibitem[14]{ML01-1} Leonid Makar-Limanov. Some conjectures, examples and counterexamples. \emph{Annales Polonici Mathematici}, 76(1-2):139-145, 2001.
\bibitem[15]{ML01-2} Leonid Makar-Limanov. On the group of automorphisms of a surface $x^ny=p(z)$. \emph{Israel Journal of Mathematics}, 121:113-123, 2001.
\bibitem[16]{Sa64} Pierre Samuel. \emph{On unique factorization domains}. Lecture notes, \href{http://www.math.tifr.res.in/~publ/ln/tifr30.pdf}{AlgebraNote 30}, 1998.

\end{thebibliography}
\end{document}